\documentclass[12pt]{amsart}

\input xy
\usepackage[english]{babel}
\usepackage[latin1]{inputenc}
\usepackage{graphicx}
\usepackage{amsmath}
\usepackage{amssymb}
\usepackage{amscd}
\usepackage{color}
\usepackage{oldgerm}
\usepackage{amsfonts}
\usepackage{newlfont}
\usepackage{longtable}
\usepackage{multirow}
\usepackage{xcolor}

\DeclareOldFontCommand{\rm}{\normalfont\rmfamily}{\mathrm}

\usepackage[all]{xy}

\usepackage{upref}
\def\R{\Bbb R}
\def\F{\Bbb F}

\def\S{\mathcal S}

\def\ad{\operatorname{ad}}

\def\dim{\operatorname{dim}}

\def\Ker{\operatorname{Ker}}

\def\Id{\operatorname{Id}}

\def\Im{\operatorname{Im}}

\def\hil{\operatorname{Nil}}

\def\g{\frak g}
\def\gl{\frak{gl}}
\def\h{\frak h}

\def\a{\frak{a}}

\def\b{\frak{b}}

\theoremstyle{plain}\swapnumbers

\newtheorem{Theorem}{Theorem}[section]
\newtheorem{Lemma}[Theorem]{Lemma}

\newtheorem{Cor}[Theorem]{Corollary}

\newtheorem{Remark}[Theorem]{Remark}

\title[On Quadratic Lie algebras]{On Quadratic Lie algebras containing the Heisenberg Lie algebra}

\author{R. Garc\'{\i}a-Delgado}
\address{Centro de Investigaci\'on en Matem\'aticas A. C.,  Unidad M\'erida; Yucat\'an, M\'exico, Carretera Sierra Papacal Chuburna Puerto Km 5, 97302 Sierra Papacal, Yuc.}
\email{rosendo.garciadelgado@gmail.com}


\date{\today}

\begin{document}

\maketitle

\begin{abstract}
In this work we study quadratic Lie algebras that contain the Heisenberg Lie algebra $\h_m$ as an ideal. We give a procedure for constructing these kind of quadratic Lie algebras and prove that any quadratic Lie algebra $\g$ that contains the Heisenberg Lie algebra as an ideal is constructed by using this procedure. We state necessary and sufficiency conditions to determine whether an indecomposable quadratic Lie algebra is the Heisenberg Lie algebra extended by a derivation. In addition, we state necessary and sufficiency conditions to determine whether the quotient $\g/\h_m$ admits an invariant metric and we also study the case when the nilradical of the Lie algebra $\g$ is equal to $\h_m$. 
\end{abstract} 

\section*{Introduction}

A Lie algebra $(\g,[\,\cdot\,,\cdot\,])$ over a field $\F$ is called \emph{quadratic} if it admits a non-degenerate and symmetric bilinear form $B:\g \times \g \to \F$ that is \emph{invariant} under the adjoint representation of $\g$, i.e, $B(x,[y,z])=B([x,y],z)$ for all $x,y,z \in \g$. A bilinear form satisfying these conditions is called an \emph{invariant metric} on $\g$.
\smallskip

A semisimple Lie algebra is quadratic and the Killing form is its invariant metric. However, there are non-semisimple quadratic Lie algebras, an example of this is the following. Let $(\g,[\,\cdot\,,\,\cdot\,])$ be a Lie algebra and let $\ad^{\ast}:\g \to \gl(\g^{\ast})$ be its coadjoint representation. The vector space $\g \oplus \g^{\ast}$ with the bracket defined by $[x+\zeta,y+\nu]^{\prime}=[x,y]+\ad^{\ast}(x)(\nu)-\ad^{\ast}(y)(\zeta)$ and the invariant metric $B$ defined by $B(x+\zeta,y+\nu)=\zeta(y)+\nu(x)$ for all $x,y \in \g$ and $\zeta,\nu \in \g^{\ast}$, is a non-semisimple quadratic Lie algebra because $\g^{\ast}$ is a non-zero abelian ideal. In this case the invariant metric $B$ is different from the Killing form. 
\smallskip

One of the most important results in Lie theory is the classification of semisimple Lie algebras. One property sharing by semisimple Lie algebras is the existence of an invariant metric where Killing form acts as a such. 
Once the classification of all semisimple Lie algebras is known, it can be expected that a classification of all quadratic Lie algebras is possible. However, the problem of classifying solvable quadratic Lie algebras is as hard as classifying Lie algebras in general. Thus, there is a natural interest in obtaining structure results for non-semisimple quadratic Lie algebras because these results lead us to advance in the problem of classifying non-semisimple quadratic Lie algebras.
\smallskip

There is a procedure for constructing non-semisimples quadratic Lie algebras known as \emph{double extension} (see \cite{Figueroa} or \cite{Medina}). This procedure gives a way to construct any non-semsimple quadratic Lie algebra from another of smaller dimension. In addition, Medina-Revoy theorem proves that any indecomposable non-simple quadratic Lie algebra is a double extension (see \cite{Medina}). These results have been used to classify families of quadratic Lie algebras (see for instance \cite{Elduque}, \cite{Dat}). Another efforts have been made by considering canonical ideals for quadratic Lie algebras without simple ideals (see \cite{Kath}). In \cite{Favre} the authors introduce the concept of amalgamation of quadratic Lie algebras and give a list of nilpotent quadratic Lie algebras of dimension less than 8.
\smallskip

Quadratic Lie algebras also appear in other areas of mathematics. If a Lie group admits a bi-invariant semi-Riemannian metric, then its Lie algebra endowed with the semi-Riemannian metric becomes a quadratic Lie algebra. Reciprocally, the invariant metric of a quadratic Lie algebra induces by left-translation a bi-invariant semi-Riemannian metric on any connected Lie group with the Lie algebra. In Conformal Field Theory, quadratic Lie algebras are the Lie algebras for which a Sugawara construction exists (see \cite{Figueroa}). 
\smallskip

As we mentioned above, classifying non-semisimple quadratic Lie algebras is as hard as classifying Lie algebras in general. This leads us to study families of non-semisimple quadratic Lie algebras with some interest not only in algebra, but also in geometry and physics. One of these families is that consisting of quadratic Lie algebras containing the Heisenberg Lie algebra $\h_m$ as an ideal. 
\smallskip

Throughout this manuscript, the Heisenberg Lie algebra of dimension $2m+1$ is denoted by $\h_m=V \oplus \F \hslash$, where $V$ is a $2m$-dimensional vector space endowed with a non-degenerate skew-symmetric bilinear form $\omega:V \times V \to \F$ such that the bracket $[\,\cdot\,,\,\cdot\,]_{\h_m}$ on $\h_m$ is defined by
\begin{equation}\label{cero}
\begin{split}
& [u,v]_{\h_m}=\omega(u,v)\hslash\,\qquad \text{ for all }\,u,v \in V\,\text{ and }\\
& [u+\gamma \hslash,\hslash]_{\h_m} = 0 \qquad \text{ for all }u \in V\,\text{ and }\gamma \in \F.
\end{split}
\end{equation}
It can be proved that any invariant and symmetric bilinear form on $\h_m$ degenerates on $\hslash$. However, there is an extension of $\h_m$ that does admit an invariant metric that below we describe. 
\smallskip

Let $\phi:\h_m \to \h_m$ be a linear map satisfying 
$$
\phi(\hslash)=0,\quad \phi(V) \subset V \quad \text{ and }\quad \phi|_V \text{ is invertible. }
$$
In addition suppose that $\omega(\phi(u),v)=-\omega(u,\phi(v))$ for all $u,v \in V$, i.e, $\phi$ belongs to $\mathfrak{o}(\omega)$. In the vector space $\F \phi \oplus \h_m$ define a bracket $[\,\cdot\,,\,\cdot\,]$ by $[\phi,x]=\phi(x)$ for all $x \in \h_m$ and $[u,v]=\omega(u,v) \hslash$ for all $u,v \in V$. Observe that $\phi$ is a derivation of the Heisenberg Lie algebra $\h_m$. Then, the vector space $\F \phi \oplus \h_m=\F \phi \oplus V \oplus \F \hslash$ becomes a quadratic Lie algebra with invariant metric $B$ defined by
\begin{equation}\label{doble extension para Heisenberg}
\begin{split}
& B(u,v)=\omega(\phi^{-1}(u),v)\quad \text{ for all }u,v \in V,\\
& B(u,\phi)=B(u,\hslash)=0\quad \text{ for all }u \in V,\\
& B(\phi,\hslash)=1\quad \text{ and }\quad B(\phi,\phi)=B(\hslash,\hslash)=0. 
\end{split}
\end{equation}
We call this quadratic Lie algebra \emph{the Heisenberg Lie algebra extended by $\phi$} and we denote it by $\h_m(\phi)$. In fact, $\h_m(\phi)$ is a double extension of the abelian quadratic Lie algebra $V$ by $\phi$ (see \ref{double extension section} below). Observe that $[\h_m(\phi),\h_m(\phi)]=\h_m$ and that $\h_m(\phi)$ is a semidirect extension of $\h_m$ by $\phi$. Any derivation of $\h_m$ whose kernel is $\F \hslash$ yields a semidirect extension of $\h_m$ that admits an invariant metric. In \cite{Ovando} it is proved that abelian subalgebras of these kind of derivations of $\h_m$ describe mechanical systems for quadratic Hamiltonians of $\R^{2m}$.  
\smallskip

In Theorem \ref{teorema-doble-ext} below we give a procedure for constructing a quadratic Lie algebra that contains the Heisenberg Lie algebra as an ideal. Before that, we develop results that elucidate the structure of such quadratic Lie algebras. In Theorem \ref{teorema f} we state necessary and sufficiency conditions to determine whether a quadratic Lie algebra is isomorphic to the Heisenberg Lie algebra extended by the derivation. In Theorem \ref{teorema ind} we state necessary and sufficiency conditions so that the quotient $\g/\h_m$ admits an invariant metric. In Theorem \ref{teorema 1} we prove that if the nilradical of $\g$ is $\h_m$, then the solvable radical $\operatorname{Rad}(\g)$ is a non-degenerate ideal and it is equal to the Heisenberg Lie algebra extended by a derivation.

\section{Heisenberg Lie algebra as an ideal}

\begin{Lemma}\label{lema-aux-0}{\sl
Let $(\g,[\,\cdot\,,\,\cdot\,],B)$ be a quadratic Lie algebra such that the Heisenberg Lie algebra $\h_m=V \oplus \F \hslash$ is an ideal of $\g$. Then,
\smallskip

\textbf{(i)} $\F \hslash \subset Z(\g)$, i.e, $[\hslash,x]=0$ for all $x \in \g$.
\smallskip

\textbf{(ii)} $\h_m \subset (\F \hslash)^{\perp}$, i.e., $B(\hslash,x)=0$ for all $x \in \h_m$.
\smallskip

\textbf{(iii)} The subspace $V$ is non-degenerate.
\smallskip

\textbf{(iv)} There exists a subspace $\a$ of $\g$ such that $\a \subset V^{\perp}$ and $\g=\a \oplus \h_m$.
\smallskip

\textbf{(v)} The bracket $[\,\cdot\,,\,\cdot\,]$ of $\g=\a \oplus V \oplus \F \hslash$ has the following decomposition
\begin{equation}\label{descripcion-corchete}
\begin{split}
\forall\,a,b \in \a,\quad & [a,b]=[a,b]_{\a}+\mu(a,b)\hslash,\\
\forall\,a \in \a,\,u \in V,\quad & [a,u]=\rho(a)(u),\\
\forall\,u,v \in V,\quad & [u,v]=\omega(u,v)\hslash.
\end{split}
\end{equation}
Where the bilinear map $(a,b) \mapsto [a,b]_{\a}$ makes $\a$ into a Lie algebra isomorphic to the quotient $\g/\h_m$, $\mu:\a \times \a \to \F$ is a skew-symmetric bilinear map and $\rho:\a \to \mathfrak{o}(\omega)$ is a representation of $\a$ where $\mathfrak{o}(\omega)=\{f \in \operatorname{End}(V)\mid \omega(f(u),v)=-\omega(u,f(v))\text{ for all }u,v \in V\}$.}
\end{Lemma}
\begin{proof}
Let $\a$ be a complementary subspace to $\h_m$ in $\g$, i.e., $\g=\a \oplus \h_m$. Since $\h_m$ is an ideal of $\g$, for every $a \in \a$ the adjoint representation $\ad(a)$ yields a linear map $R(a):\h_m \to \h_m$ defined by $[a,x]=R(x)$ for all $x \in \h_m$. From the Jacoby identity of the bracket $[\,\cdot\,,\,\cdot\,]$, it follows that the map $R(a):\h_m \to \h_m$ is a derivation of $\h_m$.
\smallskip

\textbf{(i)} We shall prove that $\F \hslash \subset Z(\g)$. Observe that the center of $\h_m$ is $\F \,\hslash$. Since the center $\F \hslash$ of $\h_m$ is invariant under any derivation of $\h_m$, then $R(a)(\F \hslash) \subset \F \hslash$ for all $a \in \a$. As $[\hslash,\h_m]=\{0\}$, then $\F\hslash$ is a one-dimensional ideal of $\g$. It is straightforward to verify that any one-dimensional ideal of a quadratic Lie algebra is contained in the center. Then, $\F \hslash \subset Z(\g)$.
\medskip

\textbf{(ii)} We shall prove that $\F h_m \subset (\F \hslash)^{\perp}$, i.e., $B(\hslash,x)=0$ for all $x \in \h_m$. Let $0 \neq u \in V$ be an arbitrary element. Because $\omega$ is non-degenerate, there exists $v \in V$ such that $[u,v]=\hslash$. Then,
$$
B(u,\hslash)=B(u,[u,v])=B([u,u],v)=0,
$$
which shows that $B(V,\hslash)=\{0\}$. In addition, as $[V,\hslash]=\{0\}$ we have
$$
B(\hslash,\hslash)=B([u,v],\hslash)=B(u,[v,\hslash])=0.
$$
Therefore, $B(\h_m,\hslash)=B(V \oplus \F \hslash,\hslash)=\{0\}$.
\medskip

\textbf{(iii)} We shall prove that the subspace $V$ is non-degenerate. Let $a \in \a$ and $u \in V$, then $[a,u]=R(a)(v) \in \h_m=V \oplus \F \hslash$. Hence, there are linear maps $\rho(a):V \to V$ and $\tau(a):V \to \F$ such that $R(a)(u)=\rho(a)(u)+\tau(a)(v)\hslash$ for all $a \in \a$ and $u \in V$.
\smallskip

As $\omega$ is a non-degenerate, for each $a \in \a$ there exists an element $L(a) \in V$ such that $\tau(a)(v)=\omega(L(a),v)$ for all $v \in V$. Then,
\begin{equation}\label{descNov2}
[a,u]=R(a)(u)=\rho(a)(u)+\omega(L(a),v)\hslash.
\end{equation}

By \textbf{(ii)} we know that $\h_m \subset (\F \hslash)^{\perp}$, then the linear map $\eta:\a \to \F$, $a \mapsto B(a,\hslash)$, is non-zero. Hence there is an element $d \in \a$ such that $B(d,\hslash)=1$, consequently $\a=\Ker(\eta) \oplus \F d$.
\smallskip

Let $a \in \a$ and $u,v \in V$. Using that $B$ is invariant and that $\h_m \subset (\F \hslash)^{\perp}$, by \eqref{descNov2} we have
\begin{equation}\label{descNov3}
B(a,[u,v])=\omega(u,v)B(a,\hslash)=B([a,u],v)=B(\rho(a)(u),v).
\end{equation}
Substituting $a$ by $d$ in \eqref{descNov3} we get $\omega(u,v)=B(\rho(d)(u),v)$. If $u$ lies in $\Ker(\rho(d))$, then $\omega(u,v)=0$ for all $v \in V$. As $\omega$ is non-degenerate, $u=0$ proving that $\rho(d):V \to V$ is invertible. Therefore,
\begin{equation}\label{descNov4}
B(u,v)=\omega(\rho(d)^{-1}(u),v)\qquad \text{ for all }u,v \in V.
\end{equation}
As $\omega$ is non-degenerate, by \eqref{descNov4} it follows that $V$ is non-degenerate.
\medskip

\textbf{(iv)} We have proved that $V$ is a non-degenerate subspace of $\g$, then $\g=V \oplus V^{\perp}$, where by \textbf{(ii)}, $\F \hslash \subset V^{\perp}$. Let $\a^{\prime}$ be a vector subspace complementary to $\F \hslash$ in $V^{\perp}$, i.e, $V^{\perp}=\a^{\prime} \oplus \F \hslash$. Then,
\begin{equation*}\label{a-ortogonal}
\g=V\oplus V^{\perp}=V \oplus (\a^{\prime} \oplus \F \hslash)=\a^{\prime} \oplus \h_m.
\end{equation*}
Therefore, from now on we assume that the complementary vector subspace $\a=\a^{\prime}$ is contained in $V^{\perp}$ and $\g=\a \oplus \h_m$.
\medskip

\textbf{(v)} We claim that for every $a \in \a$, the map $\rho(a)$ satisfies 
\begin{equation}\label{skew-symmetry}
\omega(\rho(a)(u),v)=-\omega(u,\rho(a)(v))\,\quad \text{ for all }u,v \in V.
\end{equation}
By \textbf{(i)} we have $[a,[u,v]]=\omega(u,v)[a,\hslash]=0$. On the other hand, the Jacobi identity of $[\,\cdot\,,\,\cdot\,]$ implies
$$
\aligned
\,0=[a,[u,v]]&=[[a,u],v]+[u,[a,v]]=[\rho(a)(u),v]+[u,\rho(a)(v)]\\
\,&=\omega(\rho(a)(u),v)\hslash+\omega(u,\rho(a)(v))\hslash,
\endaligned
$$
from it follows \eqref{skew-symmetry}. 
\smallskip

Since $\g=\a \oplus \h_m$, for each pair of elements $a,b \in \a$, there are unique elements $[a,b]_{\a} \in \a$ and $\Lambda(a,b) \in \h_m$ such that $[a,b]=[a,b]_{\a}+\Lambda(a,b)$. The Jacobi identity of the bracket of $\g$ implies that the pair $(\a,[\,\cdot\,,\,\cdot\,]_{\a})$ is a Lie algebra isomorphic to the quotient $\g/\h_m$. 
\smallskip

As $\h_m=V \oplus \F \hslash$, for each pair $a,b \in \a$, there are unique elements $\lambda(a,b) \in V$ and $\mu(a,b) \in \F$ such that $\Lambda(a,b)=\lambda(a,b)+\mu(a,b)\hslash$. Hence,
\begin{equation}\label{s1}
[a,b]=[a,b]_{\a}+\lambda(a,b)+\mu(a,b)\hslash=[a,b]_{\a}+\Lambda(a,b),
\end{equation}

We will prove that the map $\rho:\a \to \gl(V)$ appearing in \eqref{descNov2}, is a Lie algebra representation of $\a$. Let $a,b \in \a$ and $u \in V$. From \eqref{s1} we get
\begin{equation}\label{g9}
\begin{split}
\,[[a,b],u]&=[[a,b]_{\a},u]+[\lambda(a,b),u]\\
\,&=\rho([a,b]_{\a})(u)+\omega(L([a,b]_{\a}),u)\hslash+\omega(\lambda(a,b),u)\hslash.
\end{split}
\end{equation}
On the other hand, by the Jacobi identity we have
\begin{equation}\label{g10}
\begin{split}
\,[[a,b],u]&=[a,[b,u]]-[b,[a,u]]=[a,\rho(b)(u)]-[b,\rho(a)(u)]\\
\,&=[\rho(a),\rho(b)]_{\gl(V)}(u)\\
\,&+(\omega(\rho(a)(L(b)),u)-\omega(\rho(b)(L(a)),u))\hslash
\end{split}
\end{equation}
Comparing \eqref{g9} and \eqref{g10} we obtain
\begin{align}
\label{g11} & \rho([a,b]_{\a})=[\rho(a),b)]_{\gl(V)},\\
\label{g12} & L([a,b]_{\a})+\lambda(a,b)=\rho(a)(L(b))-\rho(b)(L(a)).
\end{align}
From \eqref{g11} we deduce that $\rho:\a \to \gl(V)$ is a representation. In fact, from \eqref{skew-symmetry} we have that $\Im(\rho)$ is contained in $\mathfrak{o}(\omega)$.
\smallskip

Consider the following bracket $[\,\cdot\,,\,\cdot\,]^{\prime}$ on $\g$:
\begin{equation}\label{productoX}
\begin{split}
\forall\, a,b \in \a,\quad & [a,b]^{\prime}=[a,b]_{\a}+\left(\mu(a,b)+\omega(L(a),L(b))\right)\hslash,\\
\forall\, a \in \a,\,v \in V, \quad  & [a,v]^{\prime}=\rho(a)(v),\\
\forall\, u,v \in V, \quad & [u,v]^{\prime}=\omega(u,v)\hslash.
\end{split}
\end{equation}
Observe that $[a,b]^{\prime}$ has no component in $V$. Let $a \in \a$ and $u \in V$. Define the map $\Phi:(\g,[\,\cdot\,,\,\cdot\,]^{\prime}) \to (\g,[\,\cdot\,,\,\cdot\,])$ by 
$$
\Phi(a)=a-L(a),\quad \Phi(u)=u \quad \text{ and }\quad \Phi(\hslash)=\hslash,
$$
where $a$  is in $\a$ and $u$ is in $V$. We affirm that $\Phi$ is an algebra isomorphism. Indeed, let $a,b \in \a$, then
\begin{equation}\label{g13}
\begin{split}
\Phi([a,b]^{\prime})&=\Phi([a,b]_{\a})+\mu(a,b)\hslash+\omega(L(a),L(b))\hslash,\\
\,&=[a,b]_{\a}-L([a,b]_{\a})+\mu(a,b)\hslash-\omega(L(a),L(b))\hslash.
\end{split}
\end{equation}
On the other hand,
\begin{equation}\label{g14}
\begin{split}
[\Phi(a),\Phi(b)]&=[a-L(a),b-L(b)]\\
\,&=[a,b]_{\a}+\lambda(a,b)+\mu(a,b)\hslash-\rho(a)(L(b))-\omega(L(a),L(b))\\
\,&+\rho(b)(L(a))+\omega(L(b),L(a))\hslash+\omega(L(a),L(b))\hslash\\
\,&=[a,b]_{\a}+\lambda(a,b)-\rho(a)(L(b))+\rho(a)(L(b))-\omega(L(a),L(b)))\hslash
\end{split}
\end{equation}
From \eqref{g12}, \eqref{g13} and \eqref{g14} we get $\Phi([a,b]^{\prime})=[\Phi(a),\Phi(b)]$. In addition, let $a \in \a$ and $v \in V$, then
$$
\Phi([a,v]^{\prime})=\Phi(\rho(a)(v))=\rho(a)(v).
$$
On the other hand,
$$
\aligned
\,[\Phi(a),\Phi(v)]&=[a-L(a),v]=[a,v]-[L(a),v]\\
\,&=\rho(a)(v)+\omega(L(a),v)\hslash-\omega(L(a),v)\hslash=\rho(a)(v).\\
\endaligned
$$
Hence $\Phi([a,v]^{\prime})=[\Phi(a),\Phi(v)]$. Since $\Phi\vert_{\h_m}=\operatorname{Id}_{\h_m}$ then $\Phi([u,v]^{\prime})=[\Phi(u),\Phi(v)]$ for all $u,v \in V$. Therefore, $\Phi$ is a Lie algebra isomorphism and by \eqref{productoX}, from now on we shall assume that the bracket $[\,\cdot\,,\,\cdot\,]$ of $\g$ can be written as
\begin{equation}\label{g15}
\begin{split}
\,\forall\, a,b \in \a, \quad \,& [a,b]=[a,b]_{\a}+\mu(a,b)\hslash,\\
\,\forall\, a \in \a,\, v \in V, \quad & [a,v]=\rho(x)(v),\\
\,\forall\, u,v \in V,\quad & [u,v]=\omega(u,v)\hslash.
\end{split}
\end{equation}
Where $(\a,[\,\cdot\,,\,\cdot\,]_{\a})$ is a Lie algebra, $\mu:\a \times \a \to \F$ is a skew-symmetric bilinear map and $\rho:\a \to \mathfrak{o}(\omega)$ is a Lie algebra representation.
\end{proof}

\subsection{Double extension}\label{double extension section}

For our purposes, we shall consider a particular case of a double extension of a quadratic Lie algebra. Let $(\S,[\,\cdot\,,\,\cdot\,]_{\S},B_{\S})$ be a quadratic Lie algebra and let $D:\S \to \S$ be a  derivation of $\S$ such that $B_{\S}(D(x),y)=-B_{\S}(x,D(y))$ for all $x,y \in \S$, i.e., $D$ is skew-symmetric with respect to $B_{\S}$. Let $\hslash$ be a symbol. In the vector space $\F D \oplus \S \oplus \F \hslash$, define the bilinear map $[\,\cdot\,,\,\cdot\,]$ by
$$
\aligned
\,\forall\,x \in \S,\quad & [D,x]=D(x),\\
\,\forall\,x,y \in \S,\quad & [x,y]=[x,y]_{\S}+B_{\S}(D(x),y)\hslash,\\
\,\forall\,x  \in \F D \oplus \S \oplus \F \hslash,\,\quad & [\hslash,x]=0.
\endaligned
$$
Then $[\,\cdot\,,\,\cdot\,]$ is a Lie bracket on $\F D \oplus \S \oplus \F \hslash$. In addition, we define the bilinear form $B$ on $\F D \oplus \S \oplus \F \hslash$ by
$$
\aligned
& B(x,y)=B_{\S}(x,y)\quad \text{ for all }x,y \in \S,\qquad B(D,\hslash)=1,\\
& B(x,D)=B(x,\hslash)=0\quad \text{ for all }x \in \S,\qquad B(D,D)=B(\hslash,\hslash)=0.
\endaligned
$$ 
Then, $B$ is an invariant metric on $\F D \oplus \S \oplus \F \hslash$. The quadratic Lie algebra on $\F D \oplus \S \oplus \F \hslash$ is known as the \textbf{double extension of $\S$ by $D$} and we denote it by $\S(D)$.

\begin{Lemma} \label{lema-aux}{\sl
Let $(\g,[\,\cdot\,,\,\cdot\,],B)$ be a quadratic Lie algebra containing the Heisenberg Lie algebra $\h_m$ as an ideal. Then, there exists a complementary vector subspace $\a$ to $\h_m$ in $\g$ satisfying \textbf{(iv)} and \textbf{(v)} of Lemma \ref{lema-aux-0}. In addition, there is a non-degenerate subspace $\S$ of $\g$ such that
\smallskip

\textbf{(i)} $\a=\mathcal{S} \oplus \F d$, where $\S^{\perp}=\F d \oplus \h_m$ and $B(d,\hslash)=1$.
\smallskip

\textbf{(ii)} $[\a,\a]_{\a} \subset \S$.
\smallskip

\textbf{(iii)} The subspace $\S$ has a quadratic Lie algebra structure.
\smallskip

\textbf{(iv)} $[\S,\h_m]=\{0\}$.
\smallskip

\textbf{(v)} There is a quadratic Lie algebra on $\F d \oplus \S \oplus \F \hslash$ isomorphic to a double extension of $\S$ by a derivation $D$ of $\S$.
\smallskip

\textbf{(vi)} There is a linear map $\sigma:\F d \oplus \S \oplus \F \hslash \to \mathfrak{o}(\omega)$ such that $\sigma(d)$ is invertible and $\sigma(x)=0$ for all $x \in \S \oplus \F \hslash$.}
\end{Lemma}
\begin{proof}

\textbf{(i)-(ii)} From Lemma \ref{lema-aux-0}.(iv)-(v), we know that there is a complementary vector subspace $\a$ to $\h_m$ in $\g$ satisfying \textbf{(iv)} and \textbf{(v)} of Lemma \ref{lema-aux-0}, i.e., $\a \subset V^{\perp}$ and the bracket $[\,\cdot\,,\,\cdot\,]$ of $\g$ has the decomposition given in \eqref{g15}.
\smallskip

Let us consider the linear map $\eta:\a \to \F$, $a \mapsto B(a,\hslash)$, that we used in the proof of Lemma \ref{lema-aux-0}.(iii). As $\h_m \subset (\F \hslash)^{\perp}$ (see Lemma \ref{lema-aux-0}.\textbf{(ii)}), then $\eta \neq 0$ and $[\a,\a]_{\a} \subset \Ker(\eta)$.
\smallskip

Let $d \in \a$ be such that $B(d,\hslash)=1$ and $\a=\Ker(\eta) \oplus \F d$. Because $\hslash$ belongs to $Z(\g)$ and $[\g,\g] \subset \Ker(\eta) \oplus V \oplus \F \hslash$, the linear map 
$$
x \mapsto x \quad \text{ for all } x \in \Ker(\eta) \oplus V \oplus \F \hslash\qquad \text{ and }\quad d \mapsto d-\frac{1}{2}B(d,d)\hslash,
$$
is an isomorphism of $\g$. Therefore, we assume that $B(d,d)=0$. As $\a=\Ker(\eta) \oplus \F d$, we define the subspace
$$
\b=\{a-B(a,d)\hslash\mid a \in \a\}=\{a-B(a,d)\hslash\mid a \in \Ker(\eta)\} \oplus \F d.
$$ 
Then $\b=\mathcal{S} \oplus \F d$ where 
$$
\aligned
& \bullet\, \S=\{a-B(a,d)\hslash\mid a \in \Ker(\eta)\},\\
& \bullet\, \g=\S \oplus \F d \oplus \h_m\,\,\text{ and }\\
& \bullet\, B(\S,\hslash)=B(\S,d)=\{0\}=B(\S,V).
\endaligned
$$
Observe that $\b$ is a complementary vector subspace to $\h_m$ in $\g$, i.e., $\g=\b \oplus \h_m$. In addition, observe that $\S \subset (\F d \oplus \h_m)^{\perp}$. 
\smallskip

We will prove that $\S$ is a non-degenerate subspace of $\g$. Let $x$ be in $\S \cap \S^{\perp}$. Then $B(x,\S)=\{0\}$ and $B(x,\F d \oplus \h_m)=\{0\}$, hence $B(x,\g)=\{0\}$, which implies that $x=0$. This proves that $\S$ is a non-degenerate subspace of $\g$ and $\S^{\perp}=\F d \oplus \h_m$.
\smallskip

The skew-symmetric bilinear map $[\,\cdot\,,\,\cdot\,]_{\b}$ defined on $\b=\S \oplus \F d$, by
$$
[a-B(a,d)\hslash,b-B(b,d)\hslash]_{\b}=[a,b]_{\a}-B([a,b]_{\a},d)\hslash\quad \text{ for all }a,b \in \a,
$$
makes $\b$ into a Lie algebra. In addition, $[\b,\b]_{\b}$ is contained in $\S$ because $[\a,\a]_{\a} \subset (\F \hslash)^{\perp}$, as $[\a,\a]-[\a,\a]_{\a} \subset \F \hslash$. 
\smallskip

We define the following bracket $[\,\cdot\,,\,\cdot\,]^{\prime \prime}$ on $\g=\S \oplus \F d \oplus \h_m$ by
\begin{equation*}\label{g15-A}
\begin{split}
\,\forall\, a,b \in \a, \quad \,& [a-B(a,d)\hslash,b-B(b,d)\hslash]^{\prime \prime}\\
\,&=[a-B(a,d)\hslash,b-B(b,d)\hslash]_{\b}+\mu(a,b)\hslash,\\
\,\forall\, a \in \a,u \in V,\quad & [a-B(a,d)\hslash,u]^{\prime \prime}=\rho(a)(u),\\
\,\forall\,u,v \in V,\,\quad & [u,v]^{\prime \prime}=\omega(u,v)\hslash.
\end{split}
\end{equation*}
Let $\Psi:\g \to (\S \oplus \F d) \oplus \h_m$ be the linear map defined by
$$
\aligned
\,\forall\,a \in \a,\quad & \Psi(a)=a-B(a,d)\hslash,\\
\,\forall\,x \in \h_m,\quad & \Psi(x)=x.
\endaligned
$$
It is straightforward to verify that $\Psi$ is a Lie algebra isomorphism between $(\g,[\,\cdot\,,\,\cdot\,])$ and $(\g,[\,\cdot\,,\,\cdot\,]^{\prime \prime})$. This isomorphism allows us make the identification $\a \leftrightarrow \b=\S \oplus \F d$. Therefore, from now on we assume that there exists a complementary vector subspace $\a$ to $\h_m$ in $\g$ such that $\a = \S \oplus \F d$, where $\S^{\perp}=\F d \oplus \h_m$, $[\a,\a]_{\a} \subset \S$ and $B(d,\hslash)=1$.
\smallskip

\textbf{(iii)} We shall prove that $B$ restricted to $\S \times \S$ is invariant. Indeed, let $x,z$ be in $\S$ and $y \in \a$. As $B(x,\hslash)=B(z,\hslash)=0$, then we have
\begin{equation}\label{skew}
B([x,y]_{\a},z)=B([x,y],z)=B(x,[y,z])=B(x,[y,z]_{\a}).
\end{equation}
Whence, $B\vert_{\S \times \S}$ is invariant under $[\,\cdot\,,\,\cdot\,]_{\a}$. By \textbf{(ii)} we know that $\S$ is a non-degenerate subspace of $\g$. Therefore, $(\S,[\,\cdot\,,\,\cdot\,]_{\S},B_{\S})$ is a quadratic Lie algebra, where $[\,\cdot\,,\,\cdot\,]_{\S}={[\,\cdot\,,\,\cdot\,]_{\a}}\vert_{\S \times \S}$ and $B_{\S}=B\vert_{\S \times \S}$.
\smallskip

\textbf{(iv)} We will prove that $[S,V]=\{0\}$. We know that there is a Lie algebra representation $\rho:\a \to \mathfrak{o}(\omega)$ such that $[a,u]=\rho(a)(u)$ for all $a \in \a=\S \oplus \F d$ and $u \in V$. In addition $\rho(d)$ is invertible (see \eqref{descNov4}).
\smallskip

Because $\rho(d)$ is invertible, then $V \subset [\a,\h_m]=\rho(\a)(\h_m)$. As $\omega$ is non-degenerate, there are $u,v \in V$ such that $[u,v]=\hslash$. Then, $V \oplus \F \hslash=\h_m=[\g,\h_m]$, hence 
$$
\h_m^{\perp}=C_{\g}(\h_m)=\{x \in \g\mid [x,y]=0\text{ for all }y \in \h_m\}.
$$
Since $\S^{\perp}=\F d \oplus \h_m$, then $\h_m \subset \S^{\perp}$, hence $\S \subset \h_m^{\perp}$. Therefore, $[\S,V]=\{0\}$. 
\smallskip

\textbf{(v)} Define the linear map $D:\S \to \S$ by $D(a)=[d,a]_{\a}$ for all $a \in \S$. Since $[\a,\a]_{\a} \subset \S$, the map $D$ is well defined. The Jacobi identity shows that $D$ is a derivation of $\S$ and from \eqref{skew} we deduce that it is skew-symmetric with respect to $B_{\S}$. Because $B(d,\hslash)=1$ and $B(d,\S)=\{0\}$, we have
$$
\aligned
0&=B(d,[a,b]_{\a})=B(d,[a,b])-\mu(a,b)=B([d,a],b)-\mu(a,b)\\
&=B_{\S}(D(a),b)-\mu(a,b),
\endaligned
$$
hence $\mu(a,b)=B_{\S}(D(a),b)$ for all $a,b \in \S$. In summary, as $\F d \oplus \F \hslash \subset \S^{\perp}$, on the vector space $\F d \oplus \S \oplus \F \hslash$ we have the following
\begin{equation}\label{ecuaciones}
\begin{split}
& [d,a]_{\a}=D(a) \quad \text{ for all }a \in \S,\\
& [a,b]=[a,b]_{\S}+B_{\S}(D(a),b)\hslash \quad \text{ for all }a,b \in \S,\\
& B(a,b)=B_{\S}(a,b),\quad B(a,d)=B(a,\hslash)=0,\quad B(d,\hslash)=1,\\
& B(d,d)=B(\hslash,\hslash)=0.
\end{split}
\end{equation}
Then $\F d \oplus \S \oplus \F \hslash$ is a double extension of $\S$ by $D$.
\smallskip

\textbf{(vi)} We define the map $\sigma:\F d \oplus \S \oplus \F \hslash \to \mathfrak{o}(\omega)$ by $\sigma(x)=\rho(x)$ for all $x \in \S \oplus \F d$ and $\sigma(\hslash)=0$. As $[\S,V]=\{0\}$, then $\sigma(x)=0$ for all $x \in \S \oplus \F \hslash$ and $\sigma(d)=\rho(d)$ is an invertible map.
\end{proof}

\begin{Lemma}\label{lema X}{\sl
Let $(\g,[\,\cdot\,,\,\cdot\,],B)$ be a quadratic Lie algebra containing the Heisenberg Lie algebra $\h_m$ as an ideal. Let $\a=\S \oplus \F d$ be the subspace of $\g$ satisfying the conditions of Lemma \ref{lema-aux} and let $D:\S \to \S$ be the derivation of $\S$ defined by $D(a)=[d,a]_{\a}$ (see \eqref{ecuaciones}). Then,
\smallskip

\textbf{(i)} $Z(\g)=(\Ker(D) \cap Z(\S)) \oplus \F \hslash$.
\smallskip

\textbf{(ii)} $[\g,\g]=(\Im (D) + [\S,\S]_{\S}) \oplus \h_m$.
\smallskip

\textbf{(iii)} The subspace $\S^{\perp}=\F d \oplus \h_m$ has a quadratic Lie algebra structure isomorphic to the Heisenberg Lie algebra extended by a derivation.

}
\end{Lemma}
\begin{proof}

\textbf{(i)-(ii)} This follows by the decomposition of the bracket $[\,\cdot\,,\,\cdot\,]$ given in \eqref{corchete-extension}.
\smallskip

\textbf{(iii)} By Lemma \ref{lema-aux}.(i), we know that $\S^{\perp}=\F d \oplus V \oplus \F \hslash$. We define $\phi:\h_m \to\h_m$ by $\phi(u)=\rho(d)(u)$ and $\phi(\hslash)=0$ for all $u \in V$. Then $\phi$ is a derivation of $\h_m$ and $\S^{\perp}$ is isomorphic to the Heisenberg Lie algebra $\h_m$ extended by the derivation $\phi$.
\smallskip


\end{proof}

In the following result we state a procedure to construct a quadratic Lie algebra that contains the Heisenberg Lie algebra as an ideal. We also prove that any quadratic Lie algebra containing the Heisenberg Lie algebra as an ideal can be constructed in this way.

\begin{Theorem}\label{teorema-doble-ext}{\sl
Let $(\S,[\,\cdot\,,\,\cdot\,]_{\S},B_{\S})$ be a quadratic Lie algebra and let $D$ be a derivation in $\mathfrak{o}(B_{\S})$. Let $\hslash$ be a symbol and in the vector space $\F D \oplus \S \oplus \F \hslash$ consider the double extension $\S(D)$ of $\S$ by $D$. Let $V$ be a $2m$-dimensional vector space equipped with a non-degenerate skew-symmetric bilinear form $\omega:V \times V \to \F$. If there exists a linear map $\sigma:\S(D) \to \mathfrak{o}(\omega)$ such that $\sigma(D)$ is invertible and $\sigma(x)=0$ for all $x \in \S \oplus \F \hslash$, then in the vector space $\g=\S \oplus \F D \oplus V \oplus \F \hslash$ there is a quadratic Lie algebra structure containing the Heisenberg Lie algebra as an ideal. Reciprocally, every quadratic Lie algebra containing the Heisenberg Lie algebra as an ideal can be constructed in this way.} 
\end{Theorem}
\begin{proof}
On the vector space $\g=\S \oplus \F D \oplus V \oplus \F \hslash$, we define the bracket $[\,\cdot\,,\,\cdot\,]$ by
\begin{equation}\label{corchete-extension}
\begin{split}
\,\forall\, a,b \in \S, \quad \,& [a,b]=[a,b]_{\S}+B_{\S}(D(a),b)\hslash,\\
\,\forall\, u \in V,\, \quad & [D,u]=\sigma(d)(u),\\
\,\forall\, u,v \in V,\quad & [u,v]=\omega(u,v)\hslash.
\end{split}
\end{equation}
Observe that $V \oplus \F \hslash$ is an ideal of $\g$ and that it is the Heisenberg Lie algebra. Consider the bilinear form $B_V$ on $V$ defined by $B_V(u,v)=\omega(\sigma(D)^{-1}(u),v)$ for all $u,v \in V$. Observe that $B_V$ is symmetric and non-degenerate. Now define the bilinear form $B$ on $\g=\S(D) \oplus V$ by 
\begin{equation}\label{metrica-extension}
B(x+u,y+v)=B_{\S(D)}(x,y)+B_V(u,v)
\end{equation}
for all $x,y \in \S(D)$ and $u,v \in V$. Then, $B$ is an invariant metric on $(\g,[\,\cdot\,,\,\cdot\,])$.
\smallskip

The proof that any quadratic Lie algebra containing the Heisenberg Lie algebra as an ideal, has the bracket as in \eqref{corchete-extension} and the invariant metric as in \eqref{metrica-extension}, is obtained from Lemma \ref{lema-aux-0}, the expression \eqref{descNov4} and Lemma \ref{lema-aux}.

\end{proof}

\begin{Theorem}\label{teorema f}{\sl
Let $\g$ be an indecomposable quadratic Lie algebra. Then $\g$ is the Heisenberg Lie algebra $\h_m$ extended by a derivation if and only if $[\g,\g]=\h_m$.}
\end{Theorem}
\begin{proof}
Suppose that $\g$ is the Heisenberg Lie algebra extended by a derivation. In the introduction we showed that $[\g,\g]=\h_m$. 
\smallskip

Now suppose that $\g$ is an indecomposable quadratic Lie algebra such that $[\g,\g]=\h_m$. By Lemma \ref{lema-aux} we can write $\g=\S \oplus \F d \oplus \h_m$, where $[\g,\g]=(\Im (D) + [\S,\S]_{\S}) \oplus \h_m$ (see Lemma \ref{lema X}.(ii)). Hence, $\Im(D)=[\S,\S]_{\S}=\{0\}$. This implies $D=0$ and that $\S$ is a non-degenerate ideal of $\g$ (see \eqref{corchete-extension}). As $\g$ is indecomposable, $\S=\{0\}$. Therefore, $\g=\F d \oplus V \oplus \F \hslash$ is the Heisenberg Lie algebra extended by a derivation.
\end{proof}

Let $\g$ be a vector space endowed with a bilinear form. For every $x \in \g$, define $B^{\flat}:\g \to \g^{\ast}$ by $B^{\flat}(x)(y)$ for all $y \in \g$. If $B$ is non-degenerate, then $B^{\flat}$ is an invertible map and we consider its inverse map denoted by $B^{\sharp}:\g^{\ast} \to \g$. Then, if $B^{\sharp}(\alpha)=x$, where $\alpha$ is in $\g^{\ast}$ and $x$ is in $\g$, then $\alpha(y)=B^{\flat}(x,y)$ for all $y \in \g$.
\smallskip

In this work we aboard the problem of determining the structure of a quadratic Lie algebra $\g$ for which its nilradical is the Heisenberg Lie algebra $\h_m$. In this case, the quotient $\g/\h_m$ turns out to be a reductive Lie algebra and therefore, is quadratic. The following result deals with this situation in a more general settling.

\begin{Theorem}\label{teorema ind}{\sl
Let $(\g,[\,\cdot\,,\,\cdot\,],B)$ be a quadratic Lie algebra. Then, the quotient $\g/\h_m$ admits an invariant metric if and only if there exists a subspace $\a$ of $\g$ such that
\smallskip

$\bullet$ $\g=\a \oplus \h_m$.
\smallskip

$\bullet$ $\a$ is a subalgebra of $\g$, i.e., $[\a,\a] \subset \a$.
}
\end{Theorem}
\begin{proof}
Suppose that there exists a subspace $\a$ such that $\g=\a \oplus \h_m$ and $[\a,\a] \subset \a$. We will find an invariant metric on $\a$. By Lemma \ref{lema-aux-0}.(ii) we assume that $\a=\S \oplus \F d$ where $B(\S,\hslash)=\{0\}$ and $B(d,\hslash)=1$. Using the same arguments as in the proof of Lemma \ref{lema-aux}.(iii), we also assume that $[\a,\a] \subset \S$ and $B(\S,d)=\{0\}$. Because $B$ restricted to $\S$ is non-degenerate, this implies that $[d,\S]=\{0\}$, hence $[d,\a]=\{0\}$. Now we define an invariant metric $B_{\a}$ on $\a$ by
$$
B_{\a}(a+\lambda d,b+\mu d)=B\vert_{\S \times \S}(a,b)+\lambda \mu\,\,\,\text{ for all }a,b \in \S\text{ and }\lambda,\mu \in \F.
$$

Now suppose that $\g/\h_m$ admits an invariant metric. We will find a complementary vector space to $\h_m$ that it is a subalgebra of $\g$.
\smallskip

Let $\a$ be a subspace complementary to $\h_m$ in $\g$, i.e, $\g=\a \oplus \h_m$. From Lemma \ref{lema-aux-0}.(v), there exists a Lie algebra structure $(\a,[\,\cdot\,,\,\cdot\,]_{\a})$ such that $\g/\h_m$ is isomorphic to $(\a,[\,\cdot\,,\,\cdot\,]_{\a})$ and the decomposition of the bracket $[\,\cdot\,,\,\cdot\,]$ of $\g$ given in \eqref{g15}, holds. In addition, from Lemma \ref{lema-aux-0}.(iv) we assume that $\a \subset V^{\perp}$. 
\smallskip

By hypothesis, $\g/\h_m$ admits an invariant metric, hence $(\a,[\,\cdot\,,\,\cdot\,]_{\a})$ is a quadratic Lie algebra. Let $B_{\a}$ be the invariant metric on $\a$.
\smallskip

\begin{Lemma}\label{lema 1}{\sl
There are linear maps $\varphi:\a \to \g$ and $\phi:\g \to \a$ such that
\smallskip

\textbf{(i)} $B_{\a}(a,b)=B(\varphi(a),b+x)$ for all $a,b \in \a$ and $x \in \g$.
\smallskip

\textbf{(ii)} $B(x,a)=B_{\a}(\phi(x),a)$ for all $a \in \a$ and $x \in \g$.
\smallskip

\textbf{(iii)} $\Ker(\phi)=\a^{\perp}$.
\smallskip

\textbf{(iv)} $\phi\circ \varphi=\Id_{\g}$.
\smallskip

\textbf{(v)} $\Im(\varphi)=\h_m^{\perp}$ and $\varphi([a,b]_{\a})=[\varphi(a),b]$ for all $a,b \in \a$.
}
\end{Lemma}
\begin{proof}

\textbf{(i)} Let $p:\g \to \a$ be the linear projection of $\g$ onto $\a$, i.e,  $p(a+x)=a$, for all $a \in \a$ and $x \in \h_m$. Let $p^{\ast}:\a^{\ast} \to \g^{\ast}$ be the dual map of $p$. Consider the map $\varphi:\a \to \g$ defined by $\varphi=B^{\sharp}\circ p^{\ast} \circ B_{\a}^{\flat}$. Then,
\begin{equation*}\label{varphi}
p^{\ast}(B_{\a}^{\flat}(a))=B^{\flat}(\varphi(a))\quad \text{ for all }a \in \a.
\end{equation*}
Hence for any $a,b \in \a$ and $x \in \h_m$ we have
\begin{equation}\label{varphi-2}
\begin{split}
B(\varphi(a),b+x)&=B^{\flat}(\varphi(a))(b+x)=p^{\ast}(B^{\flat}_{\a}(a))(b+x)\\
\,&=B^{\flat}_{\a}(a)(p(b+x))=B_{\a}(a,b)
\end{split}
\end{equation}
From we deduce that $\Im(\varphi) \subset \h_m^{\perp}$. 
\smallskip

\textbf{(ii)-(iii)} Let $\iota:\a \to \g$ be the inclusion map. Let $\iota^{\ast}:\g^{\ast} \to \a^{\ast}$ be the dual map of $\iota$. Consider the map $\phi:\g \to \a$ defined by $\phi=B_{\a}^{\sharp} \circ \iota^{\ast} \circ B^{\flat}$. Then,
\begin{equation*}\label{phi}
\iota^{\ast}(B^{\flat}(x))=B_{\a}^{\flat}(\phi(x))\quad \text{ for all }x \in \a.
\end{equation*}
Hence for any $x \in \g$ and $a \in \a$ we have
\begin{equation}\label{phi-2}
\begin{split}
B_{\a}(\phi(x),a)&=B_{\a}^{\flat}(\phi(x))(a)=\iota^{\ast}(B^{\flat}(x))(a)\\
\,&=B^{\flat}(x)(\iota(a))=B(x,a).
\end{split}
\end{equation}
Whence, $\Ker(\phi)=\a^{\perp}$. 
\smallskip

\textbf{(iv)} Let $a,b \in \a$, from \eqref{varphi-2} and \eqref{phi-2} we get
$$
B_{\a}(a,b)=B(\varphi(a),b)=B_{\a}(\phi(\varphi(a)),b).
$$
As $B_{\a}$ is non-degenerate, it follows $\phi \circ \varphi=\Id_{\a}$.
\smallskip

\textbf{(v)} Observe that $x-\varphi(\phi(x)) \in \Ker(\phi)=\a^{\perp}$ for all $x \in \g$. As $\Im(\varphi) \subset \h_m^{\perp}$ and $\a^{\perp}\cap \h_m^{\perp}=\{0\}$, we have
\begin{equation}\label{lun12-17}
\g=\Im(\varphi) \oplus \a^{\perp}.
\end{equation}
We claim that $\Im(\varphi)=\h_m^{\perp}$. Indeed, let $x$ be in $\h_m^{\perp}$. By \eqref{lun12-17}, there are $y \in \Im(\varphi)$ and $z \in \a^{\perp}$ such that $x=y+z$. Because $y \in \Im(\varphi) \subset \h_m^{\perp}$, then $z=x-y \in \a^{\perp}\cap \h_m^{\perp}=\{0\}$. Hence $x=y \in \Im(\varphi)$ and $\h_m^{\perp}=\Im(\varphi)$. This proves that $\Im(\varphi)$ is an ideal of $\g$. 
\smallskip

Now we shall prove that
\begin{equation*}\label{invariancia-phi-0}
\varphi([a,b]_{\a})=[\varphi(a),b]=[a,\varphi(b)]\quad \text{ for all }a,b \in \a.
\end{equation*}
Since $\Im(\varphi)$ is an ideal of $\g$ and $\varphi$ is injective, for each pair $a,b \in \a$ there is a unique element $x \in \a$ such that $[\varphi(a),b]=\varphi(x)$. Let $c \in \a$ be an arbitrary element, then
\begin{equation*}\label{lun12-31}
\begin{split}
B([\varphi(a),b],c)&=B(\varphi(a),[b,c])=B(\varphi(a),[b,c]_{\a})\\
\,&=B_{\a}(a,[b,c]_{\a})=B_{\a}([a,b]_{\a},c).
\end{split}
\end{equation*}
Whereas, $B([\varphi(a),b],c)=B(\varphi(x),c)=B_{\a}(x,c)$. As $B_{\a}$ is non-degenerate, we deduce that $x=[a,b]_{\a}$ and $\varphi([a,b]_{\a})=[\varphi(a),b]$.
\end{proof}

\begin{Lemma}\label{lema 2}{\sl
There are a linear maps $F,T:\a \to \a$ such that
\smallskip

\textbf{(i)} $B_{\a}(T(a),b)=B_{\a}(a,T(b))$ for all $a,b \in \a$.
\smallskip

\textbf{(ii)} $[a,b]=[a,b]_{\a}+B_{\a}(F(a),b)\hslash$ for all $a,b \in \a$.
\smallskip

\textbf{(iii)} $F$ is a derivation of $(\a,[\,\cdot\,,\,\cdot\,]_{\a})$.
\smallskip

\textbf{(iv)} $T([a,b]_{\a})=[T(a),b]_{\a}$ for all $a,b \in \a$.
\smallskip

\textbf{(v)} $T \circ F=F \circ T=\ad_{\a}(e^{\prime})$ for some element $e^{\prime} \in \a$.
\smallskip

\textbf{(vi)} $\Im(T) \subset V^{\perp}$.}
\end{Lemma}
\begin{proof}

\textbf{(i)} As $\Im(\varphi) \subset \g=\a \oplus \h_m$, for every $a \in \a$, let $T(a) \in \a$ and $S(a) \in \h_m$ be such that
\begin{equation}\label{descomposicion-varphi}
\varphi(a)=T(a)+S(a).
\end{equation}
We denote by $T:\a \to \a$ the linear map $a \mapsto T(a)$. Similarly, we denote by $S:\a \to \h_m$ the linear map $a \mapsto S(a)$.
\smallskip

We claim that $T$ satisfies $B_a(T(a),b)=B_{\a}(a,T(b))$ for all $a,b \in \a$. Indeed, by Lemma \ref{lema 1}.(i)-(v) we get
$$
B_{\a}(T(a),b)=B(T(a),\varphi(b))=B(\varphi(a),\varphi(b)).
$$ 
Similarly we prove $B_{\a}(a,T(b))=B(\varphi(a),\varphi(b))$. Hence $B_a(T(a),b)=B_{\a}(a,T(b))$.
\smallskip

\textbf{(ii)-(iii)} Let us consider the bilinear map $\mu:\a \times \a \to \F$ appearing in
$$
[a,b]=[a,b]_{\a}+\mu(a,b)\F \hslash\quad \text{ for all }a,b \in \a. 
$$
(see \eqref{g15}). As $B_{\a}$ is non-degenerate, there is a linear map $F:\a \to \a$ such that $\mu(a,b)=B_{\a}(F(a),b)$ for all $a,b \in \a$. Then,
\begin{equation}\label{a-descomposicion-corchete}
[a,b]=[a,b]_{\a}+B_{\a}(F(a),b)\hslash\quad \text{ for all }a,b \in \a.
\end{equation}
Using that $\hslash \in Z(\g)$ (see Lemma \ref{lema-aux-0}.(i)), we have
$$
[a,[b,c]]=[a,[b,c]_{\a}]_{\a}+B_a(F(a),[b,c]_{\a})\hslash\quad \text{ for all }a,b,c \in \a.
$$
The Jacobi identity implies that $F$ is a derivation of $(\a,[\,\cdot\,,\,\cdot\,]_{\a})$. 
\smallskip

\textbf{(iv)-(v)} Consider the linear map $S:\a \to \h_m$ appearing in \eqref{descomposicion-varphi}. Since $\Im(S) \subset V \oplus \F \hslash=\h_m$, there are a linear map $S_V:\a \to V$ and an element $e^{\prime} \in \a$ such that $S(a)=S_V(a)+B_{\a}(e^{\prime},a)\hslash$. By \eqref{descomposicion-varphi} it follows that $\varphi(a)=T(a)+S_V(a)+B_{\a}(e^{\prime},a)\hslash$ for all $a \in \a$. 
\smallskip

We claim that $S_V=0$. Indeed, let $a \in \a$ and $u \in V$ be arbitrary elements. From Lemma \ref{lema-aux-0}.(iv) and Lemma \ref{lema 1}.(v) we have $0=B(\varphi(a),u)=B(S_V(a),u)$. As $V$ is non-degenerate, (see Lemma \ref{lema-aux-0}.(iii)), we deduce that $S_V=0$. So, $\varphi$ can be written as
\begin{equation}\label{desc-varphi-3}
\varphi(a)=T(a)+B_{\a}(e^{\prime},a)\hslash\quad \text{ for all }a \in \a.
\end{equation}
\smallskip

Using that $\varphi([a,b]_{\a})=[\varphi(a),b]$ and \eqref{desc-varphi-3}, we get
\begin{align}
\label{T-centroide}& T([a,b]_{\a})=[a,T(b)]_{\a}=[T(a),b]_{\a}\,\,\text{ and }\\
\label{F-derivation}& T \circ F=F \circ T=\ad_{\a}(e^{\prime})\,\,\text{ for all }a,b \in \a.
\end{align}

\textbf{(vi)} Finally, as $\hslash \subset \h_m^{\perp}$ (see Lemma \ref{lema-aux-0}.(ii)) and $\Im(\varphi)=\h_m^{\perp} \subset V^{\perp}$ (see Lemma \ref{lema 1}.(v)), from \eqref{desc-varphi-3} we obtain $T(a)=\varphi(a)-B_{\a}(e^{\prime},a)\hslash \subset h_m^{\perp}$, which proves that $\Im(T) \subset h_m^{\perp}$.
\end{proof}

\begin{Lemma}\label{lema 3}{\sl
There is an element $e \in \a$ such that
\smallskip

\textbf{(i)} $a=T(\phi(a+x))+B(a+x,\hslash)e$ for all $a \in \a$ and $x \in \h_m$.
\smallskip

\textbf{(ii)} $B(e,\hslash)=1$.
\smallskip

\textbf{(iii)} $T \circ F=F \circ T=\ad_{\a}(e)$.
\smallskip

\textbf{(iv)} $T(\phi(e))=0$.
}
\end{Lemma}
\begin{proof}
The linear map $\psi:\a^{\perp} \to \h_m^{\ast}$ defined by 
$$
\psi(x)(y)=B(x,y)\quad \text{ for all } x \in \Ker(\phi)=\a^{\perp} \text{ and }y \in \h_m,
$$
is a vector space isomorphism because $B$ is non-degenerate and $\hslash \subset \h_m^{\perp}$ (see Lemma \ref{lema-aux-0}.(ii)). 
\smallskip

Let $\{u_1,\ldots,u_{2m},u_{2m+1}\}$ be a basis of $\h_m$, where $u_j$ is in $V$ for all $j \in \{1,\ldots,2m\}$ and $u_{2m+1}=\hslash$. Then there is a basis 
$$
\{e_1+w^{\prime}_1,\ldots,e_{2m+1}+w^{\prime}_{2m+1}\}\subset \a^{\perp} \subset \g=\a \oplus \h_m
$$
of $\a^{\perp}$, where $e_i$ is in $\a$ and $w^{\prime}_i$ is in $\h_m$, such that
$$
\psi(e_i+w^{\prime}_i)(u_j)=B(e_i+w^{\prime}_i,u_j)=\delta_{ij}\quad \text{ for all }i,j \in \{1,\ldots,2m+1\}.
$$

Let $e_{2m+1}=e \in \a$, then $B(e,\hslash)=1$. We shall prove the following
\begin{equation}\label{descomposicion-a}
a=T(\phi(a+x))+B(a+x,\hslash)e\quad \text{ for all }a \in \a\text{ and }x \in \h_m.
\end{equation}
Indeed, let $a,b \in \a$. Using that $\a \subset V^{\perp}$ (see Lemma \ref{lema-aux-0}.(iv)), Lemma \ref{lema 2}.(i) and \eqref{desc-varphi-3}, we get
$$
\aligned
& B_{\a}(a,b)=B(a,\varphi(b))=B(a,T(b))+B_{\a}(b,e^{\prime})B(a,\hslash)\\
&=B_{\a}(\phi(a),T(b))+B_{\a}(b,B(a,\hslash)e^{\prime})\\
&=B_{\a}(T(\phi(a))+B(a,\hslash)e^{\prime},b).
\endaligned
$$
Because $B_{\a}$ is non-degenerate, we obtain
\begin{equation}\label{comb1}
a=T(\phi(a))+B(a,\hslash)e^{\prime},\quad \text{ for all }a \in \a.
\end{equation}
On the other hand, by Lemma \ref{lema 1}.(v) and Lemma \ref{lema 2}.(i), it follows
\begin{equation*}\label{comb2}
\begin{split}
0&=B(\varphi(a),x)=B(T(a),x)=B_{\a}(T(a),\phi(x))\\
\,&=B_{\a}(a,T(\phi(x)))\quad \text{ for all }x \in \h_m.
\end{split}
\end{equation*}
Hence, $T(\phi(x))=0$ for all $x \in \h_m$, as $B_{\a}$ is non-degenerate. Using this result and the fact that $\hslash \in \h_m^{\perp}$ (see Lemma \ref{lema-aux-0}.(ii)), from \eqref{comb1} we deduce
\begin{equation}\label{descomposicion-a-2}
a=T(\phi(a+x))+B(a+x,\hslash)e^{\prime}\quad \text{ for all }a \in \a\text{ and }x \in \h_m.
\end{equation}

In particular, if we replace $a+x$ by $e+w^{\prime}_{2m+1} \in \Ker(\phi)$ in \eqref{descomposicion-a-2}, we obtain $e=e^{\prime}$, proving \eqref{descomposicion-a}. In addition we get $T \circ F=F \circ T=\ad_{\a}(e)$. If we apply $\phi$ in \eqref{comb1} we obtain
$$
\phi(a)=\phi(T(\phi(a)))+B(a,\hslash)\phi(e)\quad \text{ for all }a \in \a.
$$
In particular, for $a=e$ we get $\phi(T(\phi(e)))=0$. Hence, by Lemma \ref{lema 2}.(vi), $T(\phi(e)) \in \Ker(\phi)\cap \h_m^{\perp}=\a^{\perp}\cap \h_m^{\perp}=\{0\}$, proving that $\phi(e)$ belongs to $\Ker(T)$.
\end{proof}

\begin{Lemma}\label{lema 3-1}{\sl
$\dim \Ker(T) \leq 1$ and $\Ker(T) \subset Z(\a)$.}
\end{Lemma}
\begin{proof}
From Lemma \ref{lema 3}.(i) we deduce that $\dim \Im(T)\geq \dim \g-1$, then $\dim \Ker(T) \leq 1$. As $\Ker(T)$ is an ideal of $\a$ (see Lemma \ref{lema 2}.(iv)) and $\a$ admits an invariant metric, it follows that $\Ker(T) \subset Z(\g)$. In particular, $\phi(e)$ belongs to $Z(\a)$.
\end{proof}

\begin{Lemma}\label{lema 4}{\sl
$\,$

\textbf{(i)} For all $x \in \g$ and $a \in \a$, the following holds
\begin{equation}\label{invariancia-phi}
\phi([x,a])=[\phi(x),a]_{\a}+B(x,\hslash)F(a).
\end{equation}
\smallskip

\textbf{(ii)} $\Ker(F)=C_{\a}(e)=\{a \in \a\mid [a,e]_{\a}=0\}$.
}
\end{Lemma}
\begin{proof}
\textbf{(i)} Let $x \in \g$ and $a,b \in \a$. From Lemma \ref{lema 1}.(ii) and \eqref{a-descomposicion-corchete}, it follows
$$
\aligned
& B_{\a}(\phi([x,a]),b)=B([x,a],b)=B(x,[a,b])\\
&=B(x,[a,b]_{\a})+B(x,B_{\a}(F(a),b)\hslash)\\
&=B_{\a}([\phi(x),a]_{\a},b)+B_{\a}(B(x,\hslash)F(a),b).
\endaligned
$$
Using that $B_{\a}$ is non-degenerate, we get \eqref{invariancia-phi}. 
\smallskip

\textbf{(ii)} In particular, for $x=e$ we obtain $\phi([e,a])=[\phi(e),a]_{\a}+F(a)=F(a)$, as $\phi(e) \in Z(\a)$. Then, $F(e)=0$. By \eqref{a-descomposicion-corchete}, $[e,a]=[e,a]_{\a} \in \a$ and $F(a)=\phi([e,a]_{\a})$ for all $a \in \a$, which implies that $C_{\a}(e) \subset \Ker(F)$. Let $a \in \Ker(F)$, then 
$$
[e,a]_{\a} \in \Ker(\phi) \cap \a \subset \a^{\perp} \cap V^{\perp}.
$$
because $\Ker(\phi)=\a^{\perp}$ and $\a \subset V^{\perp}$. As $[e,a]=[e,a]_{\a}$, $[e,a]$ belongs to $(\F \hslash)^{\perp}$, then $[e,a]_{\a} \in \a^{\perp} \cap V^{\perp} \cap (\F \hslash)^{\perp}=\{0\}$ because $\g=\a \oplus V \oplus \F \hslash$. Therefore, $C_{\a}(e)=\Ker(F)$.
\end{proof}

\begin{Lemma}\label{lema 5}{\sl
$F$ is an inner derivation of $(\a,[\,\cdot\,,\,\cdot\,]_{\a})$.}
\end{Lemma}
\begin{proof}
If $T=0$, by Lemma \ref{lema 3-1} we deduce that $\dim \g=1$ and therefore, $F$ is an inner derivation. Suppose that $T \neq 0$. Let 
$$
m=X^r+\gamma_1 X^{r-1}+\ldots \gamma_1X+\gamma_0 \in \F[X],\quad \gamma_j \in \F
$$
be the minimal polynomial of $T$. Take an element $a\in\a$. From $m(T)(F)(a)=0$ and $F \circ T=T \circ F=\ad_{\a}(e)$ we get
\begin{equation}\label{inner1}
0=[T^{r-1}(e),a]+[\gamma_1 T^{r-2}(e),a]+\ldots+[\gamma_1 e,a]+\gamma_0 F(a).
\end{equation}
If $\gamma_0 \neq 0$, from \eqref{inner1} we deduce that $F$ is an inner derivation. Suppose that $\gamma_0=0$ and let $\tilde{m_1}=X^{r-1}+\gamma_1 X^{r-2}+\ldots \gamma_2 X+\gamma_1 \in \F[X]$. By \eqref{inner1} it follows $[e,\tilde{m_1}(T)(a)]=0$, which means that $\tilde{m_1}(T)(a) \in C_{\a}(e)=\Ker(F)$. Hence, from $F(\tilde{m_1}(T)(a))=0$ we obtain
\begin{equation}\label{inner2}
0=[T^{r-2}(e),a]+[\gamma_1 T^{r-3}(e),a]+\ldots+[\gamma_2 e,a]+\gamma_1 F(a).
\end{equation}
If $\gamma_1 \neq 0$, from \eqref{inner2} we deduce that $F$ is an inner derivation. Suppose that $\gamma_1=0$ and let $\tilde{m_2}=X^{r-2}+\gamma_1 X^{r-3}+\ldots \gamma_3 X+\gamma_2 \in \F[X]$. By \eqref{inner2} it follows $[e,\tilde{m_2}(T)(a)]=0$, which means that $\tilde{m_2}(T)(a) \in C_{\a}(e)=\Ker(F)$. Repeating this process, we obtain an element $\gamma_{k} \neq 0$ for which
$$
0=[T^{r-k}(e),a]+[\gamma_1 T^{r-k}(e),a]+\ldots+[\gamma_{k+1} e,a]+\gamma_k F(a).
$$
Then, $F$ is an inner derivation and $F=\ad_{\a}(c)$ for an element $c \in \a$. 
\end{proof}

As $\mu(a,b)=B_{\a}(F(a),b)$, by Lemma \ref{lema-aux-0}.(v), the bracket $[\,\cdot\,,\,\cdot\,]$ of $\g$ takes the following form
\begin{equation*}
\begin{split}
\, \forall\,a,b \in \a, \quad \,&[a,b]=[a,b]_{\a}+B_{\a}(c,[a,b]_{\a})\hslash,\\
\, \forall\,a \in \a,\, v \in V, \quad &[a,v]=\rho(x)(v),\\
\, \forall\,u,v \in V,\quad &[u,v]=\omega(u,v)\hslash.
\end{split}
\end{equation*}
In the vector space $\a \oplus \h_m$, consider the bracket $[\,\cdot\,,\,\cdot\,]^{\prime}$ defined by
\begin{equation*}
\begin{split}
\, \forall\,a,b \in \a, \quad \,&[a,b]^{\prime}=[a,b]_{\a},\\
\, \forall\,a \in \a,\, x \in \h_m, \quad &[a,x]^{\prime}=[a,x],\\
\, \forall\,x,y \in \h_m,\quad & [x,y]^{\prime}=[x,y].
\end{split}
\end{equation*}
Observe that $\a$ is a subalgebra of $(\a \oplus \h_m,[\,\cdot\,,\,\cdot\,]^{\prime})$. For $a \in \a$ and $x \in \h_m$, the linear map $a \mapsto a+B_{\a}(c,a)\hslash$ and $x \mapsto x$ gives an isomorphism between $(\a \oplus \h_m,[\,\cdot\,,\,\cdot\,]^{\prime})$ and $(\g,[\,\cdot\,,\,\cdot\,])$. 
\end{proof}

\begin{Remark}{\sl
The fact that there is a complementary subspace to $\h_m$ in $\g$ that it is a subalgebra of $\g$, depends on the structure of $\h_m$. In other cases where $\g/\mathcal{I}$ admits an invariant metric and $\mathcal{I}$ is contained in the center, there is no subspace complementary to $\mathcal{I}$ that is a subalgebra of $\g$ (see \cite{Garcia}, Example 2.16).} 
\end{Remark}

\section{Heisenberg Lie algebras as the nilradical}

\begin{Theorem}\label{teorema 1}{\sl
Let $\g$ be a quadratic Lie algebra such that $\hil(\g)=\h_m$.  Then,
\smallskip

$\bullet$ $\operatorname{Rad}(\g)$ is isomorphic to the Heisenberg Lie algebra $\h_m$ extended by a derivation.
\smallskip

$\bullet$ $\operatorname{Rad}(\g)$ is a non-degenerate ideal of $\g$. 
}
\end{Theorem}
\begin{proof}
Since $\g/\hil(\g)$ is reductive, then it is a quadratic Lie algebra. By Theorem \ref{teorema ind}, there is a vector subspace $\a$ of $\g$ such that $[\a,\a] \subset \a$ and $\g=\a \oplus \h_m$. 
\smallskip

We may assume that $\a=\S \oplus \F d$, where $[\a,\a]_{\a} \subset \S$. Let us consider the decomposition of the bracket $[\,\cdot\,,\,\cdot\,]$ of $\g$ given in \eqref{corchete-extension}. As $\a$ is reductive, it follows that $d$ belongs to the center of $\a$. Therefore, $\S$ is a non-degenerate ideal of $\g$, where $\S^{\perp}=\F d \oplus \h_m$. 
\smallskip

Now we shall prove that $\operatorname{Rad}(\g)=\F d \oplus V \oplus \F \hslash$. Observe that $\F d \oplus V \oplus \F \hslash$ is solvable because
$$
[\F d \oplus V \oplus \F \hslash,\F d\oplus V \oplus \F \hslash]=\h_m
$$
and $\h_m=\hil(\g)$ is nilpotent. Then,
\begin{equation*}\label{descNov8}
\hil(\g) \subset \operatorname{Rad}(\g) \subset \F d \oplus \h_m=\F d \oplus \hil(\g).
\end{equation*}
Whence,
$$
\dim \hil(\g) < \dim \operatorname{Rad}(\g) \leq 1+\dim \hil(\g).
$$
Therefore, $\operatorname{Rad}(\g)=\F d \oplus \h_m=\F d \oplus V \oplus \F \hslash$. By Lemma \ref{lema X}.(iii), $\operatorname{Rad}(\g)$ is isomorphic to the Heisenberg Lie algebra $\h_m$ extended by a derivation.
\end{proof}

\begin{Cor}\label{cor}{\sl
$\,$

\textbf{(i)} Let $\g$ be an indecomposable quadratic Lie algebra. If $\hil(\g)$ is the Heisenberg Lie algebra, then $\g$ is the Heisenberg Lie algebra extended by a derivation.
\smallskip

\textbf{(ii)} Let $\g$ be a quadratic solvable Lie algebra. If $\,\hil(\g)$ is the Heisenberg Lie algebra, then $\g$ is the Heisenberg Lie algebra extended by a derivation.
}
\end{Cor}

\section*{Acknowledgements}

The author acknowledges the support provided by a post-doctoral fellowship CONAHCYT grant 769309.
\smallskip

\subsection*{Disclosure statement}
The author declares that he has no relevant financial interests that relate to the research described in this paper.

\bibliographystyle{amsalpha}

\end{document}